 \documentclass[twoside,reqno,11pt]{fcaa-var}
\usepackage{graphicx}
\usepackage{epsfig}
\usepackage{amsthm}
\usepackage{amsmath}
\usepackage{latexsym}
\usepackage{amsfonts}
\usepackage{amssymb}

\newtheoremstyle{theorem}% name
  {15pt}          % space above
  {15pt}  % space below
  {\sl}  % bofy font
  {\parindent}
  {\sc}  % thm head font
  {. }   % punctuation after thm head
  { }    % space after thm head: `` ``=normal \newline=linebreak
  {}     % thm head specification
\theoremstyle{theorem}
\newtheorem{lemma}{Lemma}[section]
\newtheorem{theorem}{Theorem}[section]
\newtheorem{corollary}{Corollary}[section]
\newtheorem{proposition}{Proposition}[section]

\newtheoremstyle{defi}% name
  {15pt}          % space above
  {15pt}  % space below
  {\rm}  % bofy font
  {\parindent}     % ident - empty=no indent,  \parindent= paragraph indent
  {\sc}  % thm head font
  {. }    % punctuation after thm head
  { }    % space after thm head: `` ``=normal \newline=linebreak
  {}     % thm head specification
\theoremstyle{defi}

 \def\proofname{\indent\rm P\ r\ o\ o\ f.}
 \def\proofend{\hfill$\Box$}
 \def\qed{\hfill$\Box$} % instead of \square

 \usepackage{hyperref} % Editor will use to create hyperlinks %
 \def\theequation{\arabic{section}.\arabic{equation}}

 \usepackage{amsmath}
 \usepackage{amssymb}
 \usepackage{comment}
\newcommand{\pal}{\partial}
\newcommand{\al}{\alpha}
\newcommand{\sgn}{\operatorname{sgn}}
\def\themycopyrightfootnote{\vspace*{3pt}
  \setcounter{page}{1}
  \thispagestyle{empty}}
 \title[On solutions of linear FDEs and systems thereof]{On solutions of linear fractional differential equations and systems thereof}
 \author[Kh. Dorjgotov, H. Ochiai, U. Zunderiya]{Khongorzul Dorjgotov $^1$, Hiroyuki Ochiai $^2$, Uuganbayar Zunderiya $^3$}

\begin{document}

 \vbox to 2.5cm { \vfill }

%%% to make empty space of approx. 2.5cm %%%%%%
%%% will be replaced by Editor with the journal's and Versita logos %%%%%%%%

 \bigskip \medskip

%%%% Abstract %%%%%%%%%%%%%%%%%%%%%%%%%
\begin{abstract}
It is well-known that one-dimensional time fractional diffusion-wave equations with variable coefficients can be reduced to ordinary fractional differential equations and systems of linear fractional differential equations via scaling transformations. We then derive exact solutions to classes of linear fractional differential equations and systems thereof expressed in terms of Mittag-Leffler functions, generalized Wright functions and Fox H-functions. These solutions are invariant solutions of diffusion-wave equations obtained through certain transformations, which are briefly discussed. We show that the solutions given in this work contain previously known results as particular cases.\\
%\medskip
%{\it MSC 2010\/}: 26A33, 33E12, 34A05, 34A08  \\
\smallskip
{\it Key Words and Phrases}: Riemann-Liouville fractional derivative, fractional ordinary differential equation, generalized Wright function, Fox H-function, Mittag-Leffler function, exact solution
\end{abstract}
\maketitle
\vspace*{-16pt}
\section{Introduction}\label{sec:1}
\setcounter{section}{1}
\setcounter{equation}{0}
\setcounter{theorem}{0}
Fractional differentiation was first discussed in Leibniz's notes. Since that time, almost three centuries ago, fractional differentiation has been developed mainly as a purely theoretical field of mathematics. However, for the last several decades, the application of fractional differentiation to the mathematical modeling of physical problems has become increasingly common, due to the fact that fractional differentiation provides a useful tool for the description of memory and hereditary properties of various materials and processes \cite{Pod,Kilbas,rida}. In particular, anomalous diffusion processes in complex systems, from charge transport in amorphous semiconductors to bacterial motion, have been successfully modeled with fractional diffusion wave equations \cite{Metz}. As in the case of integer-order differential equations, various methods, including Adomian decomposition and integral and differential transforms, have been applied to the problem of solving fractional differential equations \cite{turk}.

In this work, we derive exact solutions expressed in terms of well-known special functions for fractional ordinary differential equations (FODEs) of the form
\begin{equation}\label{eq6}
\frac{d^\al}{d z^\al}\varphi(z)=\frac{a_m}{\al^m}z^m\frac{d^m}{d z^m}\varphi(z)+\frac{a_{m-1}}{\al^{m-1}}z^{m-1}\frac{d^{m-1}}{d z^{m-1}}\varphi(z)+\cdots+\frac{a_{1}}{\al}z\frac{d}{d z}\varphi(z)+a_0\varphi(z),
\end{equation}
where $\al\in\mathbb{R}_+$, $a_i\in\mathbb{R}$ $(i=0,\dots,m)$ and $a_m\neq 0,$ and for systems of the form
\begin{equation}\label{eq7}
\begin{cases}
\frac{d^\al}{dz^\al}\varphi(z) = \frac{a_{m_1}}{\al^{m_1}}z^{m_1}\frac{d^{m_1}}{dz^{m_1}}\psi(z)+\frac{a_{{m_1}-1}}{\al^{m_1-1}}z^{m_1-1}\frac{d^{m_1-1}}{dz^{m_1-1}}\psi(z)+\cdots+\frac{a_1}{\al}z\frac{d}{dz}\psi(z)+a_0\psi(z),\\
\frac{d^\al}{dz^\al}\psi(z) = \frac{b_{m_2}}{\al^{m_2}}z^{m_2}\frac{d^{m_2}}{dz^{m_2}}\varphi(z)+\frac{b_{{m_2}-1}}{\al^{m_2-1}}z^{m_2-1}\frac{d^{m_2-1}}{dz^{m_2-1}}\varphi(z)+\cdots+\frac{b_1}{\al}z\frac{d}{dz}\varphi(z)+b_0\varphi(z),
\end{cases}
\end{equation}
where $\al\in\mathbb{R}_+$, $a_i, b_j\in\mathbb{R}$ $(i=0,\dots,m_1; j=0,\dots,m_2)$ and  $a_{m_1}b_{m_2}\neq 0$. Here, for $\al\in\mathbb{R}_+$, fractional differentiation is defined in the Riemann-Liouville manner:
\begin{equation}\label{eq3}
\frac{d^\al }{d z^\al}\varphi(z):=
\begin{cases}
\frac{d^n }{d z^n}\varphi(z),\quad \mbox{ for }\al=n\in\mathbb{N},\\
\frac{1}{\Gamma(n-\al)}\frac{d^n}{d z^n}\int_0^z(z-s)^{n-\al-1}\varphi(s)ds, \mbox{ for } \al\in(n-1,n) \mbox{ with } n\in\mathbb{N}.
\end{cases}
\end{equation}
Throughout this work, we consider $n\in\mathbb{N}$ satisfying $0\leq n-1<\al<n$. 

It is interesting to consider the forms taken by \eqref{eq6} and \eqref{eq7} in the particular cases that $m=2$ and $m_1=m_2=1,$ because these are the cases most commonly considered in scientific and engineering fields. In these cases, we obtain the FODE
\begin{equation}\label{eq1}
\frac{d^\al}{d z^\al}\varphi(z)=a\varphi(z)+\frac{b}{\al}z\frac{d}{dz}\varphi(z)+\frac{c}{\al^2}z^2\frac{d^2}{d z^2}\varphi(z), \mbox{ where } a, b, c\in\mathbb{R}
\end{equation}
and the system of FODEs
\begin{equation}\label{eq2}
\begin{cases}
\frac{d^\al}{d z^\al}\varphi(z) = a_1\psi(z)+\frac{b_1}{\al}z\frac{d}{dz}\psi(z),\\
\frac{d^\al}{d z^\al}\psi(z) = a_2\varphi(z)+\frac{b_2}{\al}z\frac{d}{dz}\varphi(z),
\end{cases} \mbox{ where } a_1, a_2, b_1, b_2\in\mathbb{R}.
\end{equation}
In the case that $\al=1$ and $\varphi(z)$ takes the form $\varphi(z)=z^r e^{-\frac{1}{cz}}\phi(\frac{1}{cz})$, \eqref{eq1} reduces to Kummer's equation, 
\begin{equation*}
z^2\frac{d^2}{dz^2}\phi(z)+\left(2r-\frac{b}{c}+2-z\right)\frac{d}{dz}\phi(z)+\left(\frac{b}{c}-r-2\right)\phi(z)=0,
\end{equation*}
where $r=-\frac{b-c+\sqrt{(b-c)^2-4ac}}{2c}.$ Further, in the case that $\al=2$ and $\varphi(z)$ takes the form $\varphi(z)=(cz^2-1)^{-\frac{b-2c}{4c}}\phi(\sqrt{c}z)$, \eqref{eq1} reduces to the associated Legendre differential equation, 
\begin{equation*}
(1-z^2)\frac{d^2}{dz^2}\phi(z)-2z\frac{d}{dz}\phi(z)+\left(l(l+1)-\frac{s^2}{1-z^2}\right)\phi(z)=0,
\end{equation*}
where $l=\frac{-c+\sqrt{b^2+c^2-4ac-2bc}}{2c}$ and $s=\frac{b}{2c}-1.$   The solutions of the above equations can be expressed in terms of Kummer's function and associated Legendre functions, respectively.

Interestingly, \eqref{eq1} and \eqref{eq2} can also be obtained from the fractional diffusion-wave equation
\begin{equation}\label{eq4}
\frac{\pal^\al u}{\pal t^\al}=c(x)^2u_{x x}
\end{equation}
and the system 
\begin{equation}\label{eq5}
\begin{cases}
\frac{\pal^\al u}{\pal t^\al}=c(x)^2v_{x},\\
\frac{\pal^\al v}{\pal t^\al}=u_{x}
\end{cases}
\end{equation}
with variable diffusion coefficient $c(x)=A(x+B)^k$ or $c(x)=Ae^{k x},$ where $A,$ $B$ and $k$ are real constants. Specifically, if we transform \eqref{eq4} and \eqref{eq5} by scaling with the similarity variable $z=(x+B)^{\frac{s}{\al}}t$ (where $s$ are suitably chosen real numbers) in the case $c(x)=A(x+B)^k$ and with the similarity variable $z=e^{\frac{k}{\al}x}t$ in the case $c(x)=Ae^{k x},$ then we obtain \eqref{eq1} and \eqref{eq2}, where the constants $a,$ $b,$ $c,$ $a_1,$ $a_2,$ $b_1$ and $b_2$ are expressed in terms of $\al$, $A,$ $B$ and $k$. Thus, we can obtain  exact invariant solutions to \eqref{eq4} and \eqref{eq5} by obtaining exact solutions to \eqref{eq1} and \eqref{eq2}, respectively. Symmetry reductions of time fractional diffusion-wave equations and systems with variable diffusion coefficients and exact invariant solutions, which can be obtained using the results of the present paper, appear in works by the present authors \cite{our1,our2}.

In the following section, we briefly introduce the special functions that are used to express the solutions of \eqref{eq6} and \eqref{eq7}. In Section 3, we present the solutions of \eqref{eq1} and \eqref{eq2} in three propositions along with several technical lemmas and corollaries. The results obtained there hint at a general pattern for the solutions to \eqref{eq6} and \eqref{eq7}. In Section 4, we explicitly identify this pattern and derive exact solutions to \eqref{eq6} and \eqref{eq7} using the roots of the characteristic polynomials of the right-hand sides of \eqref{eq6} and \eqref{eq7} in  analogy to the well-known method of solving Cauchy-Euler differential equations. In Section 5, we derive exact solutions of diffusion-wave equations and systems with variable coefficients using the results obtained in Section 3 and show that our solutions correspond to known solutions in some cases. 
\section{Preliminaries}\label{sec:2}
\setcounter{section}{2}
\setcounter{equation}{0}
\setcounter{theorem}{0}
We express the solutions of \eqref{eq1} and \eqref{eq2} in terms of Mittag-Leffler functions, generalized Wright functions and Fox H-functions, which are defined as follows:
\begin{itemize}
\item The Mittag-Leffler function, 
\begin{equation*}
E_{\al,\beta}(z)=\sum_{i=0}^\infty\frac{z^i}{\Gamma(\al i+\beta)},
\end{equation*}
is defined for $z\in\mathbb{C}$ and for $\al\in\mathbb{R}_+$ and $\beta\in\mathbb{R}$ \cite{MLnom}.

\item The Wright function,
\begin{equation*}
\Psi(z;\al,\beta)=\sum_{i=0}^\infty\frac{z^i}{i!\Gamma(\alpha i+\beta)},
\end{equation*}
is an entire function for $z\in\mathbb{C}$ and for real $\al$ satisfying  $\alpha>-1$ and $\beta\in\mathbb{C}$ \cite{luchko99}.

\item The generalized Wright function,
\begin{eqnarray*}
{}_p\Psi_q\left[z\left|\begin{array}{c}
(A_i,\al_i)_{1,p}\\
(B_j,\beta_j)_{1,q}
\end{array}\right.\right] & = & \sum_{k=0}^\infty\frac{\prod\limits_{i=1}^p\Gamma(A_i+\al_i k)}{\prod\limits_{j=1}^q\Gamma(B_j+\beta_j k)}\frac{z^k}{k!},
\end{eqnarray*}
is defined for $z\in\mathbb{C},$ $p,q\in\mathbb{N}_0=\{0,1,2,\ldots\},$ $A_i,B_j\in\mathbb{C}$ and $\al_i, \beta_j\in\mathbb{R}\setminus\{0\}$ $( i=1,\ldots,p; j=1,\ldots,q).$ The generalized Wright function is absolutely convergent, and thus it is an entire function for $\Delta=
\sum\limits_{j=1}^q\beta_j-\sum\limits_{i=1}^p\alpha_i>-1$ \cite{kilb2005}.

\item The Fox H-function,
\begin{equation*}
H_{p,q}^{m,l}\left[z\biggr\vert\begin{array}{c}
(A_i, \al_i)_{1,p}\\
(B_j, \beta_j)_{1,q}
\end{array}\right]=\frac{1}{2\pi i}\int_{L}\frac{\prod\limits_{j=1}^m\Gamma(B_j-\beta_js)\prod\limits_{i=1}^l\Gamma(1-A_i+\al_is)}{\prod\limits_{i=l+1}^p\Gamma(A_i-\al_is)\prod\limits_{j=m+1}^q\Gamma(1-B_j+\beta_js)}z^{s}d s,
\end{equation*}
is defined for $z\in\mathbb{C}\setminus\{0\},$ $m,l,p,q\in\mathbb{N}_0$ with $(m,l)\neq (0,0),$ $\al_i, \beta_j\in\mathbb{R}_+$ and $A_i,B_j\in\mathbb{R}$ $(i=1,\ldots, p;j=1,\ldots,q)$. The above product is understood to be $1,$
if the number of factors in the product is zero. The contour $L$ separates the poles of the gamma functions  $\Gamma(B_j-\beta_j s)$ $(j=1,\ldots,m)$ from the poles of the gamma functions $\Gamma(1-A_i+\al_i s)$ $(i=1,\ldots,l)$. In this work, we take $L$ as $L_{\gamma+i\infty},$ a contour that extends from the point $\gamma-i\infty$ to the point $\gamma+i\infty,$ where $\gamma$ is chosen such that $L$ separates the poles as stated above. The above integral converges under the conditions \cite{Hnom}
\begin{equation*}
\mu=\sum_{i=1}^{l}\al_i-\sum_{i=l+1}^{p}\al_i+\sum_{j=1}^{m}\beta_j-\sum_{j=m+1}^{q}\beta_j>0\mbox{ and } \left\vert\arg z\right\vert<\frac{\pi\mu}{2}.
\end{equation*}
\end{itemize}
With regard to expressions for solutions to \eqref{eq6} and \eqref{eq7}, we are particularly interested in the case $l=0$ of the H-function. In this case, the H-function vanishes  exponentially for large $z$ as \cite{math}
\begin{equation}\label{eq8}
H_{p,q}^{m,0}[z]\approx O\left(\exp\left(-\nu z^{\frac{1}{\nu}}\epsilon^{\frac{1}{\nu}}\right)z^{\frac{2\delta+1}{2\nu}}\right),
\end{equation}
where $\epsilon=\prod_{i=1}^p(\al_i)^{\al_i}\prod_{j=1}^q(\beta_j)^{-\beta_j},$ $\delta=\sum_{j=1}^q B_j-\sum_{i=1}^p A_i+\frac{p-q}{2}$  and
\begin{equation*}
\nu=\sum_{j=1}^{q}\beta_j-\sum_{i=1}^{p}\al_i>0.
\end{equation*}
The following identities for H-functions are known to hold for $z>0$ \cite{Hnom}:
\begin{eqnarray}
H^{m,l}_{p,q}\left[z\biggr\vert\begin{array}{c}
(A_i,\al_i)_{1,p}\\
(B_j,\beta_j)_{1,q}
\end{array}\right] & = & H^{l,m}_{q,p}\left[\frac{1}{z}\biggr\vert\begin{array}{c}
(1-B_j,\beta_j)_{1,q}\\
(1-A_i,\al_i)_{1,p}
\end{array}\right],\label{eq9}\\
H^{m,l}_{p,q}\left[z\biggr\vert\begin{array}{c}
(A_i,\al_i)_{1,p}\\
(B_j,\beta_j)_{1,q}
\end{array}\right] & = & k H^{m,l}_{p,q}\left[z^k\biggr\vert\begin{array}{c}
(A_i,k\al_i)_{1,p}\\
(B_j,k\beta_j)_{1,q}
\end{array}\right]\mbox{ for } k>0,\label{eq10}\\
z^\sigma H^{m,l}_{p,q}\left[z\biggr\vert\begin{array}{c}
(A_i,\al_i)_{1,p}\\
(B_j,\beta_j)_{1,q}
\end{array}\right] & = & H^{m,l}_{p,q}\left[z\biggr\vert\begin{array}{c}
(A_i+\sigma\al_i,\al_i)_{1,p}\\
(B_j+\sigma\beta_j,\beta_j)_{1,q}
\end{array}\right]\mbox{ for } \sigma\in\mathbb{C}.\label{eq11}
\end{eqnarray}
The following identity is also known \cite{Hnom}:
\begin{multline}\label{eq12}
\frac{d^N}{d z^N}\left(z^{\rho-1}H^{m,l}_{p,q}\left[a z^\sigma\biggr\vert\begin{array}{c}
(A_i,\al_i)_{1,p}\\
(B_j,\beta_j)_{1,q}
\end{array}\right]\right)\\
=z^{\rho-N-1}H^{m,l+1}_{p+1,q+1}\left[a z^\sigma\biggr\vert\begin{array}{c}
(1-\rho,\sigma), (A_i,\al_i)_{1,p}\\
(B_j,\beta_j)_{1,q}, (1-\rho+N,\sigma) 
\end{array}\right],
\end{multline}
where $N\in \mathbb{N},~ a,\rho,\sigma\in\mathbb{C}$ and $\Re(\sigma)>0.$ 

Moreover, the following relations hold among the above special functions. The Mittag-Leffler functions and Wright functions can be expressed in terms of generalized Wright functions as
\begin{equation*}
E_{\al,\beta}(z) = {}_1\Psi_1\left[z\left|\begin{array}{c}
(1,1)\\
(\beta,\al)
\end{array}\right.
\right]
\end{equation*}
and
\begin{equation}\label{eq13}
\Psi\left(z;\al,\beta\right) = {}_0\Psi_{1}\left[z\left|\begin{array}{c}
-\\
(\beta,\al)
\end{array}\right.\right].
\end{equation}
Also, the generalized Wright functions can be expressed in terms of Fox H-functions as
\begin{multline}\label{eq14}
{}_p \Psi_q\left[z\biggr\vert\begin{array}{c}
(A_i,\al_i)_{1,p}\\
(B_j,\beta_j)_{1,q}
\end{array}\right]\\
=\begin{cases}
H_{p,q+1}^{1,p}\left[-z\biggr\vert\begin{array}{c}
(1-A_i,\al_i)_{1,p}\\
(0,1), (1-B_1,\beta_1),(1-B_j,\beta_j)_{2,q}
\end{array}\right],\mbox{ for } \beta_1>0\\
H_{p+1,q}^{1,p}\left[-z\biggr\vert\begin{array}{c}
(1-A_i,\al_i)_{1,p},(B_1,-\beta_1)\\
(0,1), (1-B_j,\beta_j)_{2,q}
\end{array}\right],\mbox{ for } -1<\beta_1<0,
\end{cases}
\end{multline}
where $\al_i$ $(i=1,\dots, p)$ and $\beta_j$ $(j=2,\dots,q)$ are positive real numbers \cite{main07}.

In the next section, we present exact solutions of \eqref{eq1} and \eqref{eq2} in propositions using the aforementioned special functions. We employ two approaches to prove the propositions. One is rather simple, following a term-by-term fractional differentiation of the power series representation of the Mittag-Leffler function and the generalized Wright function by using the following formula of the Riemann-Liouville derivative of power functions \cite{Pod}:
\begin{equation}\label{eq15}
\frac{d^\al}{d z^\al}z^\beta=\frac{\Gamma(1+\beta)}{\Gamma(1+\beta-\al)}z^{\beta-\al} \mbox{ for } \al>0\mbox{ and }\beta>-1.
\end{equation}
Using the fact that if the series $\sum f_j$ and the series $\sum \frac{d^\al f_j}{d z^\al}$ converge uniformly, we can take fractional derivatives term by term of a series
\begin{equation}\label{eq16}
\frac{d^\al}{d z^\al}\sum_{j=0}^\infty f_j=\sum_{j=0}^\infty\frac{d^\al f_j}{d z^\al} \quad (\mbox{with }\al>0)
\end{equation}
in its corresponding convergence region \cite{oldh}. However, if the power of the variable $z$ in power series is less than $-1,$ we cannot apply this approach. Contrastingly, we show that solutions expressed in terms of the Fox H-function, which are defined by the integral of Mellin-Barnes type,  satisfy \eqref{eq1} or \eqref{eq2}.
\section{Construction of exact solutions}\label{sec:3}
\setcounter{section}{3}
\setcounter{equation}{0}
\setcounter{theorem}{0}
As mentioned in Section 1, we express solutions to \eqref{eq1} and \eqref{eq2} using three kinds of special functions: Mittag-Leffler functions, generalized Wright functions and Fox H-functions. Which of these functions we use in  any given case depends on the right-hand side and order of the fractional derivative of \eqref{eq1} or \eqref{eq2}.
\vspace*{-12pt}
\subsection{Solutions expressed in terms of Mittag-Leffler functions}\label{subsec:3.1}
The following lemma concerns fractional derivatives of products of Mittag-Leffler functions and power functions.
\begin{lemma}\label{le1}
For arbitrary positive values of $\al,$ $\beta$ and $B,$ and for any real $a,$ the following equality holds:
\begin{equation*}
\frac{d^\al}{dz^\al}\left(z^{B-1}E_{\beta,B}(az^{\beta})\right)=
a^mz^{B+m\beta-\alpha-1}E_{\beta,B+m\beta-\alpha}(az^{\beta}).
\end{equation*}
Here, $m$ is the smallest non-negative integer such that $B+m\beta -\al-1$ is not a negative integer. 
\end{lemma}
This lemma can be proven straightforwardly by taking the fractional derivative  term by term in a series representation of the Mittag-Leffler function. For the cases specified below, we construct solutions of \eqref{eq1} and \eqref{eq2} in terms of Mittag-Leffler functions.
\begin{proposition}\label{le2}
For arbitrary $\al>0$, we have the following solutions expressed in terms of Mittag-Leffler functions.
 
(1) For $a\in \mathbb{R}$, the equation 
\begin{equation}\label{eq17}
\frac{d^\al\varphi}{d z^\al}=a\varphi,\quad z\in\mathbb{R}
\end{equation}
has a solution $\varphi(z)=\sum\limits_{k=1}^n c_kz^{\alpha-k} E_{\alpha,1+\al-k}(az^{\alpha}),$ where $c_k$ $(k=1,\ldots,n)$ are constants.

(2) For $a_1,a_2\in \mathbb{R}$, the system
\begin{equation}\label{eq18}
\begin{cases}
\frac{d^\al\varphi}{d z^\al}=a_1\psi,\\
\frac{d^\al\psi}{d z^\al}=a_2\varphi,
\end{cases}\quad z\in\mathbb{R}
\end{equation}
has a solution 
\begin{equation*}
\begin{cases}
\varphi(z) = \sum\limits_{k=1}^n c_{k,1}z^{\alpha-k} E_{
2\alpha,1+\al-k}(a_1a_2z^{2\alpha})
+a_1\sum\limits_{k=1}^nc_{k,2}z^{2\alpha-k} E_{2\alpha,1+2\al-k}(a_1a_2z^{2\alpha}),\\
\psi(z) = a_2\sum\limits_{k=1}^n c_{k,1}z^{2\alpha-k} E_{
2\alpha,1+2\al-k}(a_1a_2z^{2\alpha})+\sum\limits_{k=1}^nc_{k,2}z^{\alpha-k} E_{
2\alpha,1+\al-k}
(a_1a_2z^{2\alpha}),
\end{cases}
\end{equation*}
where $c_{k,1},c_{k,2}$ ($k=1,\ldots,n$) are constants.
\end{proposition}
\begin{proof}
From the linearity of \eqref{eq17} and \eqref{eq18}, it is sufficient to show that the single terms
\begin{equation*}
\varphi_k(z)=c_kz^{\al-k} E_{\al,1+\al-k}(az^\al)
\end{equation*}
and
\begin{equation*}
\begin{cases}
\varphi_k(z) =  z^{\alpha-k} E_{
2\alpha,1+\al-k}(a_1a_2z^{2\alpha})\\
\psi_k(z) = a_2z^{2\alpha-k} E_{
2\alpha,1+2\al-k}(a_1a_2z^{2\alpha})
\end{cases}
\end{equation*}
satisfy \eqref{eq17} and \eqref{eq18}, respectively, for $k=1,\dots,n$.  This is easily done using Lemma~\ref{le1}.
\end{proof}
Although the first assertion of Proposition~\ref{le2} was demonsrated in \cite{Pod,Kilbas}, we included it here for completeness.
\vspace*{-12pt}
\subsection{Solutions expressed in terms of generalized Wright functions}\label{sec:3.2}
Let us formulate the following contiguous relations for the generalized Wright functions. These are used below to obtain solutions of \eqref{eq1} and \eqref{eq2}.
\begin{lemma}\label{le3} 
Let us assume that the generalized Wright function is absolutely convergent, i.e., that $\Delta=\sum_{j=1}^q B_j-\sum_{i=1}^p A_i>-1.$ Then the following equalities hold for $\al \in\mathbb{R}_+$ and $a\in\mathbb{R}.$

(1) If $\beta_1>0$ and $B_1>0$, then  we have
\begin{multline*}
  \frac{d^\al}{d z^\al}\left(z^{B_1-1}{}_p\Psi_q\left[a z^{\beta_1}\biggr\vert\begin{array}{c}
(A_i,\al_i)_{1,p}\\
(B_j,\beta_j)_{1,q}
\end{array}\right]\right)
 = a^m z^{B_1+m\beta_1-1-\al}\\
 \times{}_{p+1}\Psi_{q+1}\left[az^{\beta_1}\biggr\vert\begin{array}{c}
(1,1), (A_i+m\al_i,\al_i)_{1,p}\\
(1+m,1), (B_1+m\beta_1-\al,\beta_1), (B_j+m\beta_j,\beta_j)_{2,q}
\end{array}\right],\quad z\in \mathbb{R},
\end{multline*}
where $m$ is the smallest non-negative integer such that $B_1+m\beta_1 -\al-1$ is not a negative integer.

(2) For $\sigma\in\mathbb{R}\setminus\{0\}$ and $R\in\mathbb{R},$ the following equality holds
\begin{multline*}
\left(\frac{1}{\al}z\frac{d}{d z}+R\right)\left(z^{\frac{A_1\sigma}{\alpha_1}-\alpha R}{}_p\Psi_q\left[az^{\sigma}\biggr\vert\begin{array}{c}
(A_i,\al_i)_{1,p}\\
(B_j,\beta_j)_{1,q}
\end{array}\right]\right)
\\=
\frac{\sigma}{\alpha_1\al} z^{\frac{A_1\sigma}{\alpha_1}-\alpha R}{}_p\Psi_q\left[az^{\sigma}\biggr\vert\begin{array}{c}
(A_1+1,\al_1), (A_i,\al_i)_{2,p}\\
(B_j,\beta_j)_{1,q}
\end{array}\right].
\end{multline*}
\end{lemma}
\begin{proof}
To prove the first assertion, let us write the function whose fractional derivative we are taking, $\varphi(z),$  as 
\begin{equation*}
\varphi(z)=z^{B_1-1}{}_p\Psi_q\left[a z^{\beta_1}\biggr\vert\begin{array}{c}
(A_i,\al_i)_{1,p}\\
(B_j,\beta_j)_{1,q}
\end{array}\right].
\end{equation*}
Then, taking the Riemann-Liouville derivative using \eqref{eq15}, we obtain
\begin{equation*}
\frac{d^\al\varphi}{d z^\al} = \sum_{i=0}^\infty\frac{\Gamma(A_1+\al_1 i)\cdots\Gamma(A_p+\al_p i)a^{i}}{\Gamma(B_1-\al+\beta_1 i)\Gamma(B_2+\beta_2 i)\cdots\Gamma(B_q+\beta_q i)i!}z^{B_1-1-\alpha+\beta_1i},
\end{equation*}
which equals
\begin{equation*}
\frac{d^\al\varphi}{d z^\al} = \sum_{i=0}^\infty\frac{\Gamma(A_1+\al_1 (i+m))\cdots\Gamma(A_p+\al_p (i+m))a^{i+m}z^{B_1-1-\al+\beta_1 i+\beta_1 m}}{\Gamma(B_1-\al+\beta_1 m+\beta_1 i)\Gamma(B_2+\beta_2 m+\beta_2 i)\cdots\Gamma(B_q+\beta_q m+\beta_q i)(i+m)!},
\end{equation*}
where we have employed the condition on the integer $m$. The first assertion of the lemma can be proved through multiplying both the numerator and the denominator of the last expression by $i!.$ 

The second assertion of the lemma can be proved through straightforward calculation with the help of the identity $(A_1+\al_1 i)\Gamma(A_1+\al_1 i)=\Gamma(A_1+1+\al_1 i)$ for $i=0,1,2,\dots$.
\end{proof}
We obtain the following corollary from the first assertion of Lemma \ref{le3} in the case $A_1=\al_1=1.$
\begin{corollary}\label{cor1}
When $B_1>0,$ $\beta_1>0$ and $A_1=\al_1=1,$ the fractional derivative of product of power function and generalized Wright function ${}_p\Psi_q$ is 
\begin{eqnarray*}
 & & \frac{d^\al}{d z^\al}\left(z^{B_1-1}{}_p\Psi_q\left[a z^{\beta_1}\biggr\vert\begin{array}{c}
(1,1), (A_i,\al_i)_{2,p}\\
(B_j,\beta_j)_{1,q}
\end{array}\right]\right)\\
& = & a^m z^{B_1+m\beta_1-1-\al}{}_{p}\Psi_{q}\left[az^{\beta_1}\biggr\vert\begin{array}{c}
(1,1), (A_i+m\al_i,\al_i)_{2,p}\\
(B_1+m\beta_1-\al,\beta_1), (B_j+m\beta_j,\beta_j)_{2,q}
\end{array}\right],
\end{eqnarray*}
where $m$ is the smallest non-negative integer such that $B_1+m\beta_1 -\al-1$ is not a negative integer. 
\end{corollary}
Before moving on to the formulation of exact solutions of \eqref{eq1} and \eqref{eq2}, we introduce the following notation.

Let us consider the case that $c=0$ and $b\neq 0$ in \eqref{eq1}. Then writing $-\frac{a}{b}$ as $\bar{s},$ we can rewrite the right-hand side of \eqref{eq1} as
\begin{equation*}
a\varphi+\frac{b}{\al}z\frac{d\varphi}{dz}=b\left(\frac{1}{\al}z\frac{d}{dz}-\bar{s}\right)\varphi.
\end{equation*}
Then, in the case of \eqref{eq2}, assuming $b_1b_2\neq 0$ and introducing the quantities 
\begin{equation*}
\tilde{s}_1=-\frac{a_1}{b_1},\quad \tilde{s}_2=-\frac{a_2}{b_2},
\end{equation*}
we rewrite the right-hand side of \eqref{eq2} as follows:
\begin{eqnarray*}
a_1\psi+\frac{b_1}{\al}z\frac{d\psi}{d z}=b_1\left(\frac{1}{\al}z\frac{d}{dz}-\tilde{s}_1\right)\psi,\\
a_2\varphi+\frac{b_2}{\al}z\frac{d\varphi}{d z}=b_2\left(\frac{1}{\al}z\frac{d}{dz}-\tilde{s}_2\right)\varphi.
\end{eqnarray*}
Now, let us assume $c\neq 0.$ Then the characteristic equation of the right-hand side of \eqref{eq1} is
\begin{equation}\label{eq19}
s^2+\left(\frac{b}{c}-\frac{1}{\alpha}\right)s+\frac{a}{c}=0.
\end{equation}
We write the determinant and roots of \eqref{eq19} as $D=\frac{1}{\al^2}-\frac{2b}{\al c}+\frac{b^2}{c^2}-\frac{4a}{c}$  and  $s_{1,2}=\frac{1}{2}\left(\frac{1}{\al}-\frac{b}{c}\pm\sqrt{D}\right),$ respectively.
Then we can rewrite the right-hand side of \eqref{eq1} in the factorized differential form 
\begin{equation}\label{eq20}
a\varphi+\frac{b}{\al}z\frac{d\varphi}{dz}+\frac{c}{\al^2}z^2\frac{d^2\varphi}{d z^2}=c\left(\frac{1}{\al}z\frac{d}{dz}-s_1\right)\left(\frac{1}{\al}z\frac{d}{dz}-s_2\right)\varphi.
\end{equation}
This notation is useful for at least two reasons. First, it reveals the uniformity in given solutions of \eqref{eq1} and \eqref{eq2} with different orders of fractional derivatives. In particular, in the case $c\neq 0,$ we can avoid a tedious computation by simply rewriting the right-hand side of \eqref{eq1} in factorized operator form. Second, using this notation, we can easily generalize \eqref{eq1} and \eqref{eq2} into cases with higher-order derivatives and obtain solutions thereof. We will discuss this generalization in the next section. 

We now formulate the solutions of \eqref{eq1} and \eqref{eq2} as follows.
\begin{proposition}\label{le4}
We have the following solutions expressed in terms of the generalized Wright function.

(1) For $\al>1$ and $a,b\in \mathbb{R}$ with $b\neq 0,$ the equation
\begin{equation}\label{eq21}
\frac{d^\al\varphi}{d z^\al}=a\varphi+\frac{b}{\al}z\frac{d\varphi}{dz},\quad z\in\mathbb{R}
\end{equation}
has as a solution 
\begin{equation*}
\varphi(z)=\sum\limits_{k=1}^n c_k z^{\alpha-k} {}_2\Psi_1\left[bz^{\alpha}\left|\begin{array}{c}
\left(1-\frac{k}{\alpha}-\bar{s},1\right),(1,1)
\\
(1+\al-k,\alpha)
\end{array}\right.\right],
\end{equation*}
where $\bar{s}=-\frac{a}{b},$ and $c_k$ $(k=1,\dots,n)$ are constants.

(2) For $\al>2$ and $a,b,c\in \mathbb{R}$ with $c\neq 0,$ the equation
\begin{equation}\label{eq22}
\frac{d^\al\varphi}{d z^\al} = a\varphi+\frac{b}{\alpha}z \frac{d\varphi}{dz}+\frac{c}{\al^2}z^2\frac{d^2\varphi}{dz^2},\quad z\in\mathbb{R}
\end{equation}
has as a solution 
\begin{equation*}
\varphi(z)=\sum_{k=1}^n c_kz^{\al-k}{}_3\Psi_1\left[cz^\al\left|\begin{array}{c}
          \left(1-\frac{k}{\al}-s_1,1\right),\left(1-\frac{k}{\al}-s_2,1\right),(1,1)\\
          (1+\al-k,\al)
          \end{array}\right.\right],
\end{equation*}
where $s_{1,2}=\frac{1}{2}\left(\frac{1}{\al}-\frac{b}{c}\pm\sqrt{D}\right)$, $D=\frac{1}{\al^2}-\frac{2b}{\al c}+\frac{b^2}{c^2}-\frac{4a}{c}$, and $c_k$ $(k=1,\ldots,n)$ are constants.

(3) For $\al>1$ and $a_1,a_2,b_1,b_2\in \mathbb{R}$ with $b_1b_2\neq 0,$ the system 
\begin{equation}\label{eq23}
\begin{cases}
\frac{d^\al\varphi}{d z^\al} = a_1\psi+\frac{b_1}{\alpha}z \frac{d\psi}{dz},\\
\frac{d^\al\psi}{d z^\al} = a_2\varphi+\frac{b_2}{\alpha} z\frac{d\varphi}{dz},
\end{cases}\quad z\in\mathbb{R}
\end{equation}
has as a solution 
\begin{align*}
\varphi(z)=&\sum_{k=1}^{n}c_{k,1}z^{\alpha-k} {}_3\Psi_1\left[4b_1 b_2 z^{2\al}\biggr\vert\begin{array}{c}
\left(1-\frac{k}{2\al}-\frac{\tilde{s}_1}{2},1\right),\left(\frac{1}{2}-\frac{k}{2\al}-\frac{\tilde{s}_2}{2},1\right),(1,1)\\
(1+\al-k,2\alpha)
\end{array}\right]\\
 & +2b_1\sum_{k=1}^n c_{k,2}z^{2\alpha-k} {}_3\Psi_1\left[4b_1 b_2 z^{2\al}\biggr\vert\begin{array}{c}\left(\frac{3}{2}-\frac{k}{2\al}-\frac{\tilde{s}_1}{2},1\right),\left(1-\frac{k}{2\al}-\frac{\tilde{s}_2}{2},1\right),(1,1)\\
(1+2\al-k,2\alpha)
\end{array}\right],\\
\psi(z)=&2b_2\sum_{k=1}^n c_{k,1}z^{2\alpha-k} {}_3\Psi_1\left[4b_1 b_2 z^{2\al}\biggr\vert\begin{array}{c}
\left(1-\frac{k}{2\al}-\frac{\tilde{s}_1}{2},1\right),\left(\frac{3}{2}-\frac{k}{2\al}-\frac{\tilde{s}_2}{2},1\right),(1,1)\\
(1+2\al-k,2\alpha)
\end{array}\right]\\
 & +\sum_{k=1}^n c_{k,2}z^{\alpha-k} {}_3\Psi_1\left[4b_1 b_2 z^{2\al}\biggr\vert\begin{array}{c}
\left(\frac
{1}{2}-\frac{k}{2\al}-\frac{\tilde{s}_1}{2},1\right),\left(1-\frac{k}{2\al}-\frac{\tilde{s}_2}{2},1\right),(1,1)\\
(1+\al-k,2\alpha)
\end{array}\right],
\end{align*}
where $\tilde{s}_1=-\frac{a_1}{b_1},$ $\tilde{s}_2=-\frac{a_2}{b_2},$ and $c_{k,1},$ $c_{k,2}$ $(k=1,\ldots,n)$ are constants.
\end{proposition}
\proof
The proof  can be carried out similarly in all three cases using Lemma \ref{le3} and Corollary \ref{cor1}. For this reason, we present proofs only for the second and third cases. As in Proposition \ref{le2}, from the linearity of \eqref{eq22}, it is sufficient to show that a single summand,
\begin{equation*}
\varphi_k(z)=z^{\al-k}{}_3\Psi_1\left[cz^\al\left|\begin{array}{c}
          \left(1-\frac{k}{\al}-s_1,1\right),\left(1-\frac{k}{\al}-s_2,1\right),(1,1)\\
          (1+\al-k,\al)
          \end{array}\right.\right],\mbox{ where } 1\leq k\leq n,
\end{equation*}
of the solution $\varphi(z)$ satisfies \eqref{eq22}. Because $1+\al-k>0$ for any $k,$ by Corollary \ref{cor1}, we have the following identity for the left-hand side of \eqref{eq22}:
\begin{equation}\label{eq381}
\frac{d^\al\varphi_k}{d z^\al} = c z^{\al-k}{}_3\Psi_1\left[c z^\al\biggr\vert\begin{array}{c}
(2-\frac{k}{\al}-s_1,1) , (2-\frac{k}{\al}-s_2,1) , (1,1)\\
(1+\al-k,\al)
\end{array}\right].
\end{equation}
Then, by virtue of \eqref{eq20} and the second assertion of Lemma \ref{le3}, the right-hand side of \eqref{eq22} becomes 
\begin{multline*}
c\left(\frac{1}{\al}z\frac{d}{dz}-s_1\right)\left(\frac{1}{\al}z\frac{d}{dz}-s_2\right)\varphi_k\\
=
c\left(\frac{1}{\al}z\frac{d}{dz}-s_1\right)\left( z^{\alpha-k} {}_3\Psi_1\left[cz^{\alpha}\left|\begin{array}{c}
\left(1-\frac{k}{\alpha}-s_1,1\right),\left(2-\frac{k}{\alpha}-s_2,1\right),(1,1)\\
(1+\al-k,\alpha)
\end{array}\right.\right]
\right),
\end{multline*}
which is equal to the left-hand side of \eqref{eq381}.

In the third assertion, we only need to show that the summands
\begin{equation*}
\begin{cases}
\varphi_{k,1}(z)=z^{\alpha-k} {}_3\Psi_1\left[4b_1 b_2 z^{2\al}\biggr\vert\begin{array}{c}
\left(1-\frac{k}{2\al}-\frac{\tilde{s}_1}{2},1\right),\left(\frac{1}{2}-\frac{k}{2\al}-\frac{\tilde{s}_2}{2},1\right),(1,1)\\
(1+\al-k,2\alpha)
\end{array}\right],\\
\psi_{k,1}(z)=2b_2z^{2\alpha-k} {}_3\Psi_1\left[4b_1 b_2 z^{2\al}\biggr\vert\begin{array}{c}
\left(1-\frac{k}{2\al}-\frac{\tilde{s}_1}{2},1\right),\left(\frac{3}{2}-\frac{k}{2\al}-\frac{\tilde{s}_2}{2},1\right),(1,1)\\
(1+2\al-k,2\alpha)
\end{array}\right]
\end{cases}
\end{equation*}
and
\begin{equation*}
\begin{cases}
\varphi_{k,2}(z)=2b_1z^{2\alpha-k} {}_3\Psi_1\left[4b_1 b_2 z^{2\al}\biggr\vert\begin{array}{c}\left(\frac{3}{2}-\frac{k}{2\al}-\frac{\tilde{s}_1}{2},1\right),\left(1-\frac{k}{2\al}-\frac{\tilde{s}_2}{2},1\right),(1,1)\\
(1+2\al-k,2\alpha)
\end{array}\right],\\
\psi_{k,2}(z)=z^{\alpha-k} {}_3\Psi_1\left[4b_1 b_2 z^{2\al}\biggr\vert\begin{array}{c}
\left(\frac
{1}{2}-\frac{k}{2\al}-\frac{\tilde{s}_1}{2},1\right),\left(1-\frac{k}{2\al}-\frac{\tilde{s}_2}{2},1\right),(1,1)\\
(1+\al-k,2\alpha)
\end{array}\right]
\end{cases}
\end{equation*}
satisfy \eqref{eq23} for $k=1,\dots,n$. For this purpose, it is sufficient to show that $(\varphi_{k,1}(z),\psi_{k,1}(z))$ satisfies \eqref{eq23}. Then, the proof that $(\varphi_{k,2}(z),\psi_{k,2}(z))$ satisfies \eqref{eq23} follows from the symmetry property of the system. After applying Corollary \ref{cor1} with $m=1$ for $\varphi_{k,1}$ and $m=0$ for $\psi_{k,1},$ we obtain the following expressions for the left-hand side of \eqref{eq23}:
\begin{equation}\label{eq3_8}
\begin{cases}
\displaystyle{\frac{d^\al\varphi_{k,1}}{dz^\al}}=4b_1b_2z^{2\al-k}{}_3\Psi_1\left[4b_1b_2z^{\al-k}\biggr\vert\begin{array}{c}
(1,1), \left(2-\frac{k}{2\al}-\frac{\tilde{s}_1}{2},1\right), \left(\frac{3}{2}-\frac{k}{2\al}-\frac{\tilde{s}_2}{2},1\right)\\
(1+2\al-k,2\al)
\end{array}\right],\\
\displaystyle{\frac{d^\al\psi_{k,1}}{dz^\al}}=2b_2z^{2\al-k}{}_3\Psi_1\left[4b_1b_2z^{\al-k}\biggr\vert\begin{array}{c}
(1,1), \left(1-\frac{k}{2\al}-\frac{\tilde{s}_1}{2},1\right), \left(\frac{3}{2}-\frac{k}{2\al}-\frac{\tilde{s}_2}{2},1\right)\\
(1+2\al-k,2\al)
\end{array}\right].
\end{cases}
\end{equation}
Finally, applying the second assertion of Lemma \ref{le3} to the right-hand side of \eqref{eq23}, with $R=\frac{a_1}{b_1}$ and  $\sigma=2\al$ for the first equation and $R=\frac{a_2}{b_2}$ and  $\sigma=2\al$ for the second equation, we obtain identically \eqref{eq3_8}.
\proofend
\vspace*{-12pt}
\subsection{Solutions expressed in terms of Fox H-functions}\label{sec3.3}
Let us first present the following technical lemma on the fractional differentiation of Fox H-functions with an argument raised to negative power.
\begin{lemma}\label{le5}
Let $\nu=\sum_{j=1}^{q}\beta_j-\sum_{i=1}^{p}\al_i> 0,$ $\mu=\sum\limits_{j=1}^{m}\beta_j-\sum\limits_{j=m+1}^{q}\beta_j-\sum\limits_{i=1}^{p}\alpha_i>0$, $\alpha\in \mathbb{R}_+$ and $a\in \mathbb{R}\setminus\{0\}$.  Then the following equalities hold.

(1) For $a>0$,
\begin{multline*}
\frac{d^\al}{d z^\al}H_{p,q}^{m,0}\left[az^{-\alpha_p}\biggr\vert\begin{array}{c}
(A_i,\alpha_i)_{1,p-1}, (1, \alpha_p)\\
(B_j,\beta_j)_{1,q}
\end{array}\right]\\
=z^{-\al}H_{p,q}^{m,0}\left[az^{-\alpha_p}\biggr\vert\begin{array}{c}
(A_i,\alpha_i)_{1,p-1}, (1-\al, \alpha_p)\\
(B_j,\beta_j)_{1,q}
\end{array}\right],\quad z>0.
\end{multline*}

(2) If $m\ge 1,$ then
\begin{multline*}
\left(\frac{\beta_1}{\alpha_p}z\frac{d}{dz}+B_1\right)H_{p,q}^{m,0}\left[az^{-\alpha_p}\biggr\vert\begin{array}{c}
(A_i,\alpha_i)_{1,p-1}, (1, \alpha_p)\\
(B_j,\beta_j)_{1,q}
\end{array}\right]\\
=H_{p,q}^{m,0}\left[az^{-\alpha_p}\biggr\vert\begin{array}{c}
(A_i,\alpha_i)_{1,p-1}, (1, \alpha_p)\\
(B_1+1,\beta_1),(B_j,\beta_j)_{2,q}
\end{array}\right].
\end{multline*}
\end{lemma}
\proof
By virtue of the asymptotic expression \eqref{eq8}, the fractional derivative of the function\\ $$\varphi(z)=H_{p,q}^{m,0}\left[az^{-\alpha_p}\biggr\vert\begin{array}{c}
(A_i,\alpha_i)_{1,p-1}, (1, \alpha_p)\\
(B_j,\beta_j)_{1,q}
\end{array}\right]$$ is well defined. Substituting $s=zu$ into the definition of the Riemann-Liouville derivative \eqref{eq3}, we obtain
\begin{multline}\label{eq24}
\frac{d^\al \varphi(z)}{dz^\al}  =  \frac{d^n}{d z^n}\frac{z^{n-\al}}{\Gamma(n-\al)}\int_0^1(1-u)^{n-\al-1}
\frac{1}{2\pi i}\\
\times\int_{L_{\gamma+i\infty}}\frac{\prod\limits_{j=1}^m\Gamma(B_j-\beta_js)}{\prod\limits_{i=1}^{p-1}\Gamma(A_i-\al_is)\Gamma(1-\al_p s)\prod\limits_{j=m+1}^q\Gamma(1-B_j+\beta_js)}\left(a(zu)^{-\alpha_p}\right)^{s} d s d u.
\end{multline}
The well-known formula for the beta function expressed in terms of the gamma function
\begin{equation*}%\label{14}
\int_0^1 x^{a-1}(1-x)^{b-1}d x= \frac{\Gamma(a)\Gamma(b)}{\Gamma(a+b)}, \mbox{ where } \Re(a), \Re(b)>0,
\end{equation*}
applies to our case, in which it becomes
\begin{equation}\label{eq25}
\int_0^1(1-u)^{n-\al-1}u^{-\alpha_p s}d u=\frac{\Gamma(n-\al)\Gamma(1-\alpha_p s)}{\Gamma(n-\al- \alpha_p s+1)},
\end{equation}
choosing suitable $\gamma<0$ for the contour $L_{\gamma+i\infty}$ in \eqref{eq24}. Next, interchanging the order of the integrals in \eqref{eq24} and using \eqref{eq25}, we obtain
\begin{equation}\label{eq26}
\frac{d^\al\varphi}{d z^\al}=\frac{d^n}{d z^n}z^{n-\al}H_{p,q}^{m,0}\left[a z^{-\alpha_p}\biggr\vert\begin{array}{c}
(A_i,\al_i)_{1,p-1}, (n-\al+1,\alpha_p)\\
(B_j, \beta_j)_{1,q}
\end{array}\right].
\end{equation}
The first assertion is then proved by applying \eqref{eq9} and \eqref{eq12} to \eqref{eq26}.

To prove the second assertion, let us take the derivative of the H-function using \eqref{eq12} with $N=1.$ This yields
\begin{eqnarray*}
 \lefteqn{\frac{d}{d z}H_{p,q}^{m,0}\left[az^{-\alpha_p}\biggr\vert\begin{array}{c}
(A_i,\al_i)_{1,p-1}, (1,\alpha_p)\\
(B_j,\beta_j)_{1,q}
\end{array}\right]}\\
& =& z\frac{d}{d z}H_{q,p}^{0,m}\left[\frac{z^{\alpha_p}}{a}\biggr\vert\begin{array}{c}
(1-B_j,\beta_j)_{1,q}\\
(1-A_i,\al_i)_{1,p-1}, (0,\alpha_p)
\end{array}\right]\\
&=&  H_{q+1,p+1}^{0,m+1}\left[\frac{z^{\alpha_p}}{a}\biggr\vert\begin{array}{c}
(0,\alpha_p), (1-B_j,\beta_j)_{1,q}\\
(1-A_i,\al_i)_{1,p-1}, (0,\alpha_p), (1,\alpha_p)
\end{array}\right].
\end{eqnarray*}
Here, note that the common term $(0,\al_p)$ is canceled out in the last expression, and thus, by virtue of \eqref{eq9}, we have
\begin{equation*}
 z\frac{d}{d z}H_{p,q}^{m,0}\left[az^{-\alpha_p}\biggr\vert\begin{array}{c}
(A_i,\al_i)_{1,p-1}, (1,\alpha_p)\\
(B_j,\beta_j)_{1,q}
\end{array}\right] = H_{p,q}^{m,0}\left[az^{-\alpha_p}\biggr\vert\begin{array}{c}
(A_i,\al_i)_{1,p-1}, (0,\alpha_p)\\
(B_j,\beta_j)_{1,q}
\end{array}\right].
\end{equation*}
The quantity in the second assertion is calculated as follows:
\begin{eqnarray*}
 \lefteqn{ \left(\frac{\beta_1}{\al_p}z\frac{d}{d z}+B_1\right)H_{p,q}^{m,0}\left[az^{-\alpha_p}\biggr\vert\begin{array}{c}
(A_i,\al_i)_{1,p-1}, (1,\alpha_p)\\
(B_j,\beta_j)_{1,q}
\end{array}\right]}\\
& = & \frac{1}{2\pi i}\int_{L_{\gamma+i\infty}}\left[\frac{\beta_1}{\al_p}\frac{1}{\Gamma(-\alpha_p s)}+\frac{B_1}{\Gamma(1-\alpha_p s)}\right]\\
&&\times\frac{\Gamma\left(B_1-\beta_1s\right)\prod\limits_{j=2}^m\Gamma(B_j-\beta_j s)}{\prod\limits_{j=1}^{p-1}\Gamma(A_j-\al_j s)\prod\limits_{j=m+1}^q\Gamma(1-B_j+\beta_j s)}(az^{-\alpha_p})^sds.
\end{eqnarray*}
A straightforward simplification completes the proof.

An alternative proof of the second assertion can be obtained by using the formulas for particular derivatives of Fox-H functions given in \cite{Hnom}.
\proofend\\
Unlike the previously presented solutions expressed in terms of Mittag-Leffler functions and generalized Wright functions, we present the solutions  expressed in terms of Fox-H functions for $z>0$. 

Henceforth, we use sign function for $x\in\mathbb{R}$ defined as $$\sgn(x)=\begin{cases}
1,\quad &x>0,\\
-1,\quad &x<0.
\end{cases}$$
\begin{proposition}\label{le6}
For the following cases, we have solutions expressed in terms of Fox H-functions.

(1) For $0<\al<1$ and $a,b\in \mathbb{R}$ with $b>0,$ the equation
\begin{equation}\label{eq27}
\frac{d^\al\varphi}{d z^\al}=a\varphi+\frac{b}{\al}z\frac{d\varphi}{dz},\quad z>0
\end{equation}
has as a solution
\begin{equation*}
\varphi(z)=c_1 H_{1,1}^{1,0}\left[\frac{z^{-\alpha}}{b}\biggr\vert\begin{array}{c}
\left(1,\alpha\right)\\
(-\bar{s},1)
\end{array}\right],
\end{equation*}
where $\bar{s}=-\frac{a}{b},$ and $c_1$ is a constant.

(2) For $0<\al<2$ and $a,b,c\in \mathbb{R}$ with $c>0,$ the equation
\begin{equation}\label{eq28}
\frac{d^\al\varphi}{d z^\al}=a\varphi+\frac{b}{\al}z\frac{d\varphi}{dz}+\frac{c}{\al^2}z^2\frac{d^2\varphi}{dz^2},\quad z>0
\end{equation}
has as a solution
\begin{equation*}
\varphi(z)=c_1 H_{1,2}^{2,0}\left[\frac{z^{-\alpha}}{c}\biggr\vert\begin{array}{c}
\left(1,\alpha\right)\\
\left(-s_1,1\right), \left(-s_2,1\right)
\end{array}\right],
\end{equation*}
where $s_{1,2}=\frac{1}{2}\left(\frac{1}{\al}-\frac{b}{c}\pm\sqrt{D}\right)$, $D=\frac{1}{\al^2}-\frac{2b}{\al c}+\frac{b^2}{c^2}-\frac{4a}{c},$ and $c_1$ is a constant.

(3) For $0<\al<1$ and $a_1,a_2,b_1,b_2\in \mathbb{R}$ with $b_1 b_2>0,$ the system 
\begin{equation}\label{eq29}
\begin{cases}
\frac{d^\al\varphi}{d z^\al} = a_1\psi+\frac{b_1}{\alpha}z \frac{d\psi}{dz},\\
\frac{d^\al\psi}{d z^\al} = a_2\varphi+\frac{b_2}{\alpha} z\frac{d\varphi}{dz},
\end{cases}
\quad z>0
\end{equation}
has as a solution
\begin{equation*}
\begin{cases}
\varphi(z) = c_1\sgn(b_1)  H_{1,2}^{2,0}\left[\displaystyle{\frac{z^{-2\alpha}}{4b_1 b_2}}\biggr\vert\begin{array}{c}
(1,2\alpha)\\
\left(\frac{1}{2}-\frac{\tilde{s}_1}{2},1\right), \left(-\frac{\tilde{s}_2}{2},1\right)
\end{array}\right],\\
\psi(z) = c_1 \sqrt{\frac{b_2}{b_1}} H_{1,2}^{2,0}\left[\displaystyle{\frac{z^{-2\alpha}}{4b_1 b_2}}\biggr\vert\begin{array}{c}
(1,2\alpha)\\
\left(-\frac{\tilde{s}_1}{2},1\right), \left(\frac{1}{2}-\frac{\tilde{s}_2}{2},1\right)
\end{array}\right],
\end{cases}
\end{equation*}
where $\tilde{s}_1=-\frac{a_1}{b_1},~ \tilde{s}_2=-\frac{a_2}{b_2},$ and $c_1$ is a constant.
\end{proposition}
\proof
Analogously to the proof of Proposition \ref{le4}, the three assertions of this proposition can be proved in a similar manner by using Lemma \ref{le5}. For this reason, we consider only the second and third assertions with $c_1=1,$ without loss of generality. 

Let us consider the second assertion of the proposition. Because the convergence condition of H-functions holds (i.e. $\mu=2-\alpha>0$), we can apply the first assertion of Lemma \ref{le5}. We thereby obtain
\begin{equation*}
\frac{d^\al}{d z^\al}H_{1,2}^{2,0}\left[\frac{z^{-\alpha}}{c}\biggr\vert\begin{array}{c}
\left(1,\alpha\right)\\
\left(-s_1,1\right), \left(-s_2,1\right)
\end{array}\right]=z^{-\al}H_{1,2}^{2,0}\left[\frac{z^{-\alpha}}{c}\biggr\vert\begin{array}{c}
\left(1-\al,\alpha\right)\\
\left(-s_1,1\right), \left(-s_2,1\right)
\end{array}\right]
\end{equation*}
for the left-hand side of \eqref{eq28}, which is further simplified by virtue of \eqref{eq10} into the form
\begin{equation*}
\frac{d^\al}{d z^\al}H_{1,2}^{2,0}\left[\frac{z^{-\alpha}}{c}\biggr\vert\begin{array}{c}
\left(1,\alpha\right)\\
\left(-s_1,1\right), \left(-s_2,1\right)
\end{array}\right]=cH_{1,2}^{2,0}\left[\frac{z^{-\alpha}}{c}\biggr\vert\begin{array}{c}
\left(1,\alpha\right)\\
\left(-s_1+1,1\right), \left(-s_2+1,1\right)
\end{array}\right].
\end{equation*}
The right-hand side of \eqref{eq28} is obtained in analogy to the second assertion of Proposition~\ref{le4}, by using \eqref{eq20} and interchanging the parameters $(-s_1,1)$ and $(-s_2,1)$ in the solution $\varphi(z)$ in accordance with the second assertion of Lemma~\ref{le5}.

Next, considering the third assertion, for the sake of compatibility with the second assertion of Lemma~\ref{le5}, we rewrite the right-hand side of \eqref{eq29} as follows:
\begin{equation*}
\begin{cases}
a_1\psi+\frac{b_1}{\alpha}z \frac{d\psi}{dz}=2b_1\left(\frac{1}{2\al}z\frac{d}{dz}+\frac{a_1}{2b_1}\right)\psi,\\
a_2\varphi+\frac{b_2}{\alpha} z\frac{d\varphi}{dz}=2b_2\left(\frac{1}{2\al}z\frac{d}{dz}+\frac{a_2}{2b_2}\right)\varphi.
\end{cases}
\end{equation*}
Because the convergence condition $\mu=1-\alpha>0$ holds for both $\varphi(z)$ and $\psi(z)$, we can apply the first assertion of Lemma \ref{le5} and \eqref{eq10}. We thereby obtain
\begin{equation*}
\begin{cases}
\displaystyle{\frac{d^\al\varphi}{dz^\al}} =2\sgn(b_1)\sqrt{b_1b_2} H_{1,2}^{2,0}\left[\displaystyle{\frac{z^{-2\alpha}}{4b_1 b_2}}\biggr\vert\begin{array}{c}
(1,2\alpha)\\
\left(1-\frac{\tilde{s}_1}{2},1\right), \left(\frac{1}{2}-\frac{\tilde{s}_2}{2},1\right)
\end{array}\right],\\
\displaystyle{\frac{d^\al\psi}{dz^\al}} = 2\sgn(b_2)b_2H_{1,2}^{2,0}\left[\displaystyle{\frac{z^{-2\alpha}}{4b_1 b_2}}\biggr\vert\begin{array}{c}
(1,2\alpha)\\
\left(\frac{1}{2}-\frac{\tilde{s}_1}{2},1\right), \left(1-\frac{\tilde{s}_2}{2},1\right)
\end{array}\right].
\end{cases}
\end{equation*}
Applying the second assertion of Lemma \ref{le5} with $B_1=-\frac{\tilde{s}_1}{2}$ for the first equation and $B_1=-\frac{\tilde{s}_2}{2}$ for the second equation, we obtain the desired form.
\proofend\\
We now show that for a special case of \eqref{eq28}, we have solutions expressed in terms of Wright functions.
\begin{corollary}\label{le7}
Let the determinant of \eqref{eq19} be $D=\frac{1}{\al^2}-\frac{2b}{\al c}+\frac{b^2}{c^2}-\frac{4a}{c}=\frac{1}{4}$, and suppose $c\neq 0$. 

(1) For $0<\alpha<2$ and $c> 0$, \eqref{eq28} has a solution of the following form:
\begin{equation*}
\varphi(z)=c_1z^{\frac{1}{2}\left(\frac{1}{\al}-\frac{b}{c}+\frac{1}{2}\right)\al}\Psi\left(-\frac{2z^{-\frac{\al}{2}}}{\sqrt{c}};-\frac{\al}{2},\frac{1}{2}\left(\frac{3}{\al}-\frac{b}{c}+\frac{1}{2}\right)\al
\right).
\end{equation*}

(2) For $\alpha>2$, \eqref{eq22} has a solution of the following form:
\begin{equation*}
\varphi(z)=\sum_{k=1}^{n}c_kz^{\al-k}{}_2\Psi_1\left(\frac{cz^\alpha}{4}\left\vert\begin{array}{c}
\left(\frac{3}{2}-\frac{2k+1}{\alpha}+\frac{b}{c},2\right),(1,1)\\
(1+\alpha-k,\alpha)
\end{array}\right)\right..
\end{equation*}
\end{corollary}
\proof
To prove the first assertion, we need to show that $\varphi(z)$ corresponds to the solution given in  the second assertion of Proposition~\ref{le6}. When $D=\frac{1}{4}$ and $0<\al<2$, the roots of the characteristic equation \eqref{eq19} become $s_1=\frac{1}{2}\left(\frac{1}{\alpha}-\frac{b}{c}+\frac{1}{2}\right)$ and $s_2=s_1-\frac{1}{2}.$ In this case, the solution given in the second assertion of Proposition~\ref{le6} is
\begin{equation*}
\tilde{\varphi}(z) = c_1H_{1,2}^{2,0}\left[\frac{z^{-\alpha}}{c}\biggr\vert\begin{array}{c}
\left(1,\alpha\right)\\
\left(-s_1,1\right),\left(-s_1+\frac{1}{2},1\right)
\end{array}\right].
\end{equation*}
Then, applying the duplication formula for the gamma function
\begin{equation*}
\Gamma\left(-s_1-s\right)\Gamma\left(-s_1-s+\frac{1}{2}\right)=\sqrt{\pi}2^{1+2\left(s_1+s\right)}\Gamma\left(-2s_1-2s\right)
\end{equation*}
to $\tilde{\varphi}(z),$ it becomes
\begin{eqnarray*}
\tilde{\varphi}_1(z) & = & c_1\sqrt{\pi}2^{1+2s_1}H_{1,1}^{1,0}\left[\frac{4z^{-\alpha}}{c}\biggr\vert\begin{array}{c}
\left(1,\alpha\right)\\
\left(-2s_1,2\right)
\end{array}\right]\\
& = & c_1\sqrt{\pi}4^{s_1}H_{1,1}^{1,0}\left[\frac{2z^{-\frac{\al}{2}}}{\sqrt{c}}\biggr\vert\begin{array}{c}
\left(1,\frac{\al}{2}\right)\\
(-2s_1,1)
\end{array}\right]\\
& = & c_1\sqrt{\pi}c^{s_1}z^{\frac{\al}{2}\left(\frac{1}{\al}-\frac{b}{c}+\frac{1}{2}\right)}H_{1,1}^{1,0}\left[\frac{2z^{-\frac{\al}{2}}}{\sqrt{c}}\biggr\vert\begin{array}{c}
\left(\frac{\al}{2}\left(\frac{3}{\al}-\frac{b}{c}+\frac{1}{2}\right),\frac{\al}{2}\right)\\
(0,1)
\end{array}\right]\\
& = & c_1\sqrt{\pi}c^{s_1}z^{\frac{\al}{2}\left(\frac{1}{\al}-\frac{b}{c}+\frac{1}{2}\right)}{}_0\Psi_1\left[-\frac{2z^{-\frac{\al}{2}}}{\sqrt{c}}\biggr\vert\begin{array}{c}
-\\
\left(\frac{\al}{2}\left(\frac{3}{\al}-\frac{b}{c}+\frac{1}{2}\right),-\frac{\al}{2}\right)
\end{array}\right]\\
& = & c_1\sqrt{\pi}c^{s_1}z^{\frac{\al}{2}\left(\frac{1}{\al}-\frac{b}{c}+\frac{1}{2}\right)}\Psi\left(-\frac{2z^{-\frac{\al}{2}}}{\sqrt{c}};-\frac{\al}{2},\frac{\al}{2}\left(\frac{3}{\al}-\frac{b}{c}+\frac{1}{2}\right)\right).
\end{eqnarray*}
The second through fifth equalities here follow from \eqref{eq10}, \eqref{eq11}, \eqref{eq14} and \eqref{eq13}, respectively. It is thus seen that $\tilde{\varphi}(z)=\varphi(z).$

We can prove the second assertion by applying the duplication formula for the gamma function 
\begin{equation*}
\Gamma\left(1-\frac{k}{\al}-s_1+i\right)\Gamma\left(1-\frac{k}{\al}-s_1+i+\frac{1}{2}\right)=\sqrt{\pi}2^{\frac{2k+1}{\al}-\frac{1}{2}-i}\Gamma\left(2-\frac{2k}{\al}-2s_1+2i\right)
\end{equation*}
to the solution given in the second assertion of Proposition \ref{le4}.
\proofend\\

To this point, we have presented several exact solutions of \eqref{eq1} and \eqref{eq2}. These solutions are classified according to the kind of special functions used to express them. Now, we discuss the second advantage of rewriting the right-hand sides of \eqref{eq1} and \eqref{eq2} in factorized differential operator form. Specifically, we show that utilizing this form, we are able to generalize our treatment of \eqref{eq1} and \eqref{eq2} to the cases of \eqref{eq6} and \eqref{eq7} with arbitrary integers $m,$ $m_1$ and $m_2.$ As a result, we find that the solutions to \eqref{eq21} and \eqref{eq22} of Proposition~\ref{le4} and the solutions to \eqref{eq27} and \eqref{eq28} of Proposition~\ref{le6} can all be represented in a unified manner by a single general formula. 
\section{Generalization of \eqref{eq1} and \eqref{eq2} to higher-order derivatives}
We seek a solution of the following FODE with an $m$th-order Cauchy-Euler differential operator on the right-hand side:
\begin{equation}\label{eq30}
\frac{d^\al\varphi}{d z^\al}=\frac{a_m}{\al^m}z^m\frac{d^m\varphi}{d z^m}+\frac{a_{m-1}}{\al^{m-1}}z^{m-1}\frac{d^{m-1}\varphi}{d z^{m-1}}+\cdots+\frac{a_{1}}{\al}z\frac{d\varphi}{d z}+a_0\varphi,\quad z>0,
\end{equation}
where $a_i$ $(i=0,\dots,m)$ are real numbers and $a_m\neq 0.$ 
We represent the right-hand side of \eqref{eq30} by $P(\varphi).$ Then, from
\begin{equation*}
P(z^s)=\left(a_0+\sum_{i=1}^m a_i\prod_{j=0}^{i-1}\left(\frac{s}{\al}-\frac{j}{\al}\right)\right)z^s,
\end{equation*}
we see that the characteristic polynomial of $P$ is
\begin{equation}\label{eq32}
\tilde{P}(s)=a_0+\sum_{i=1}^m a_i\prod_{j=0}^{i-1}\left(s-\frac{j}{\al}\right).
\end{equation}
Let $s_1,s_2,\dots, s_m$ be the roots of the characteristic polynomial $\tilde{P}(s)$. Then, we can rewrite the right-hand side of \eqref{eq30} as
\begin{equation*}
P(\varphi)=a_m\prod_{i=1}^m\left(\frac{1}{\al}z\frac{d}{d z}-s_i\right)\varphi.
\end{equation*}

Now, generalizing the results of the previous section, we formulate the following theorem. The proof can be carried out analogously to the proofs of Proposition~\ref{le4} and Proposition~\ref{le6}.
\begin{theorem}\label{th8}
(1) If $0<\alpha<m$ and $a_m>0,$ then \eqref{eq30} has the following as a solution:
\begin{equation*}
\varphi(z)=c_1 H_{1,m}^{m,0}\left[\frac{z^{-\alpha}}{a_m}\biggr\vert\begin{array}{c}
(1,\alpha)\\
\left(-s_1,1\right), \left(-s_2,1\right),\dots, \left(-s_m,1\right)
\end{array}\right].
\end{equation*}

(2) If $\alpha > m,$ then \eqref{eq30} has as a solution
\begin{equation*}
\varphi(z)=\sum_{k=1}^n c_k z^{\al-k}{}_{m+1}\Psi_1\left[a_m z^\al\biggr\vert\begin{array}{c}
\left(1-\frac{k}{\al}-s_1,1\right),\dots, \left(1-\frac{k}{\al}-s_m,1\right),(1,1)\\
(1+\al-k,\al)
\end{array}\right],
\end{equation*}
where $c_k$ ($k=1,\ldots,n$) are arbitrary constants.
\end{theorem}
It is now clear that the solutions given in the first and second assertions of Proposition~\ref{le4} can be expressed by using only the first assertion of Theorem \ref{th8} and that the solutions of the first and second assertion of Proposition~\ref{le6} can be expressed by using only the second assertion of Theorem \ref{th8} by taking $m=1$ and $m=2,$ respectively.

In a similar manner, we can generalize the system \eqref{eq2} as 
\begin{equation}\label{eq33}
\begin{cases}
\frac{d^\al\varphi}{dz^\al} = \frac{a_{m_1}}{\al^{m_1}}z^{m_1}\frac{d^{m_1}\psi}{dz^{m_1}}+\frac{a_{m_1-1}}{\al^{m_1-1}}z^{m_1-1}\frac{d^{m_1-1}\psi}{dz^{m_1-1}}+\cdots+\frac{a_1}{\al}z\frac{d\psi}{dz}+a_0\psi,\\
\frac{d^\al\psi}{dz^\al} = \frac{b_{m_2}}{\al^{m_2}}z^{m_2}\frac{d^{m_2}\varphi}{dz^{m_2}}+\frac{b_{{m_2}-1}}{\al^{{m_2}-1}}z^{{m_2}-1}\frac{d^{{m_2}-1}\varphi}{dz^{{m_2}-1}}+\cdots+\frac{b_1}{\al}z\frac{d\varphi}{dz}+b_0\varphi,
\end{cases}
\quad z>0,
\end{equation}
where $a_i$ $(i=1,\dots,m_1)$ and $b_j$ $(j=1,\dots,m_2)$ are real numbers and $a_{m_1}b_{m_2}\neq 0$.
The characteristic polynomials of the right-hand sides of the first and second equations of this system are
\begin{equation}\label{eq34}
P_1(s)=a_0+\sum_{i=1}^{m_1}a_i\prod_{j=0}^{i-1}\left(s-\frac{j}{\al}\right),\qquad P_2(s)=b_0+\sum_{i=1}^{m_2} b_i\prod_{j=0}^{i-1}\left(s-\frac{j}{\al}\right).
\end{equation}
We write the roots of the characteristic polynomials $P_1(s)$ and $P_2(s)$ as $s_1,s_2,\dots, s_{m_1}$ and $s_{{m_1}+1},s_{m_1+2},\dots, s_{m_1+{m_2}},$ respectively. Then, we can rewrite the right-hand sides of the equations in \eqref{eq33} as
$$
\sum_{i=0}^{m_1}\frac{a_i}{\alpha^i}z^i\frac{d^{i}\psi}{dz^{i}}= 2^{m_1}a_{m_1}\prod_{i=1}^{m_1}\left(\frac{1}{2\al}z\frac{d}{dz}-\frac{s_i}{2}\right)\psi(z)
$$
and
$$
\sum_{i=0}^{m_2}\frac{b_i}{\alpha^i}z^i\frac{d^{i}\varphi}{dz^{i}}=2^{m_2}b_{m_2}\prod_{i=m_1+1}^{m_1+{m_2}}\left(\frac{1}{2\al}z\frac{d}{dz}-\frac{s_i}{2}\right)\varphi(z).
$$

The following result concerns solutions of \eqref{eq33}. The proof can be carried out analogously to the proofs of Proposition~\ref{le4} and Proposition~\ref{le6}. 
\begin{theorem}\label{th9}
Here, we use $m$ to represent  $m_1+m_2$ and $A$ to represent $2^{m_1+m_2}a_{m_1}b_{m_2}.$

(1) If $0<\alpha<\frac{m}{2}$ and $a_{m_1}b_{m_2}>0,$ then the system \eqref{eq33} has the following as a solution:
\begin{eqnarray*}
\varphi(z) & = & c_1\sgn(b_{m_2}) H_{1,m}^{m,0}\left[\frac{z^{-2\alpha}}{A}\biggr\vert\begin{array}{c}
(1,2\alpha)\\
\left(-\frac{s_i}{2}+\frac{1}{2},1\right)_{1,{m_1}},\left(-\frac{s_i}{2},1\right)_{{m_1}+1,m}
\end{array}\right],\\
\psi(z) & = & c_12^{\frac{{m_2}-{m_1}}{2}}\sqrt{\frac{b_{m_2}}{a_{m_1}}}H_{1,m}^{m,0}\left[\frac{z^{-2\alpha}}{A}\biggr\vert\begin{array}{c}
(1,2\alpha)\\
\left(-\frac{s_i}{2},1\right)_{1,{m_1}},\left(-\frac{s_i}{2}+\frac{1}{2},1\right)_{{m_1}+1,m}
\end{array}\right].
\end{eqnarray*}

(2) If $\alpha>\frac{m}{2},$ then the system \eqref{eq33} has the following as a solution:
\begin{eqnarray*}
\varphi(z)&=&\sum_{k=1}^{n}c_{k,1}z^{\alpha-k} \varphi_{k1}(z)+2^{m_1} a_{m_1}\sum_{k=1}^n c_{k,2}z^{2\alpha-k}\varphi_{k2}(z),\\
\psi(z)&=&2^{m_2}b_{m_2}\sum_{k=1}^n c_{k,1}z^{2\alpha-k} \psi_{k1}(z) +\sum_{k=1}^n c_{k,2}z^{\alpha-{k}}\psi_{k2}(z), 
\end{eqnarray*}
where 
\begin{eqnarray*}
\varphi_{k1}(z)&=&{}_{m+1}\Psi_1\left[A z^{2\al}\biggr\vert\begin{array}{c}
\left(1-\frac{k}{2\al}-\frac{s_i}{2},1\right)_{1,{m_1}},\left(\frac{1}{2}-\frac{k}{2\al}-\frac{s_i}{2},1\right)_{{m_1}+1,m},(1,1)\\
(1+\al-k,2\alpha)
\end{array}\right],\\
\varphi_{k2}(z)&=&{}_{m+1}\Psi_1\left[A z^{2\al}\biggr\vert\begin{array}{c}\left(\frac{3}{2}-\frac{k}{2\al}-\frac{s_i}{2},1\right)_{1,{m_1}},\left(1-\frac{k}{2\al}-\frac{s_i}{2},1\right)_{{m_1}+1,m},(1,1)\\
(1+2\al-k,2\alpha)
\end{array}\right],\\
\psi_{k1}(z)&=&{}_{m+1}\Psi_1\left[A z^{2\al}\biggr\vert\begin{array}{c}
\left(1-\frac{k}{2\al}-\frac{s_i}{2},1\right)_{1,{m_1}},\left(\frac{3}{2}-\frac{k}{2\al}-\frac{s_i}{2},1\right)_{{m_1}+1,m},(1,1)\\
(1+2\al-k,2\alpha)
\end{array}\right],\\
\psi_{k2}(z)&=&{}_{m+1}\Psi_1\left[A z^{2\al}\biggr\vert\begin{array}{c}
\left(\frac
{1}{2}-\frac{k}{2\al}-\frac{s_i}{2},1\right)_{1,{m_1}},\left(1-\frac{k}{2\al}-\frac{s_i}{2},1\right)_{{m_1}+1,m},(1,1)\\
(1+\al-k,2\alpha)
\end{array}\right],
\end{eqnarray*}
and $c_1,$ $c_{k,1}$ and $c_{k,2}$ ($k=1,\dots,n$) are constants.
\end{theorem}
We, thus, see that the third assertion of Proposition~\ref{le4} and the third assertion of Proposition~\ref{le6} correspond to a particular case ${m_1}={m_2}=1$ of Theorem \ref{th9}. 
\section{Solutions to a class of fractional linear partial differential equations and systems thereof.}\label{sec:5}
In this section, we demonstrate the application of the propositions presented in this paper by providing exact solutions of generalizations of \eqref{eq4} and \eqref{eq5}. The solutions obtained in this section reduce to a previously known solution in a particular case.
\subsection{Solutions to a  fractional linear evolution equation}
Let us consider the following linear fractional evolution equation with variable coefficients: 
\begin{multline}\label{eq36}
\frac{\pal^\al u(t,x)}{\pal t^\al}=a_m(x+b)^p\frac{\pal^m u(t,x)}{\pal x^m}+a_{m-1}(x+b)^{p-1}\frac{\pal^{m-1} u(t,x)}{\pal x^{m-1}}\\
+\cdots+a_1(x+b)^{p-m+1}\frac{\pal u(t,x)}{\pal x}+a_0(x+b)^{p-m}u(t,x),\quad \alpha>0,\quad x>-b,\quad t>0.
\end{multline}
Here $a_i$ $(i=0,\dots,m)$, $b$ and $p$ are real numbers and $a_m>0.$
The infinitesimal scaling symmetries of \eqref{eq36} are
\begin{equation*}
X_1=u\frac{\pal}{\pal u},\qquad X_2=(x+b)\frac{\pal}{\pal x}+\frac{m-p}{\al}t\frac{\pal}{\pal t}.
\end{equation*}
The invariant solution corresponding to the generator $X=aX_1+X_2$ ($a\in \mathbb{R}$) is
\begin{equation}\label{eq37}
u=(x+b)^a\varphi(z),\mbox{ with } z=t(x+b)^{\frac{p-m}{\al}},
\end{equation}
where $\varphi(z)$ solves the following reduced FODE:
\begin{equation*}
\frac{d^\al\varphi}{dz^\al}=\frac{a_m\left(m-p\right)^m}{\alpha^m}z^m\frac{d^m\varphi}{dz^m}+\frac{\bar{a}_{m-1}}{\alpha^{m-1}}z^{m-1}\frac{d^{m-1}\varphi}{dz^{m-1}}+\cdots+\frac{\bar{a}_1}{\alpha}z\frac{d\varphi}{dz}+\bar{a}_0,\quad z>0. 
\end{equation*}
Here, the parameters $\bar{a}_{0},\bar{a}_{1},\dots,\bar{a}_{m-1}$ depend on  $a_1,a_2,\dots,a_m,$ $\al,$ $m$ and $p.$

Now, let us take a closer look at the case $m=2$:
\begin{equation}\label{eq38}
\frac{\pal^\al u}{\pal t^\al}=a_2(x+b)^p u_{x x}+a_1(x+b)^{p-1}u_x+a_0(x+b)^{p-2}u,\mbox{ for }\alpha>0,~ x>-b \mbox{ and } t>0,
\end{equation}
with $a_2> 0.$ For the case of $a_0=0$, solutions of \eqref{eq38} were found in \cite{Metz94} through Laplace transformation method and were expressed by Fox-H functions. Now, generalizing the result of \cite{Metz94}, let us find solutions of \eqref{eq38} for any values of $a_0$.
\begin{comment}
The solutions that we obtain now will coincide to the solutions obtained in [] when $a_0=0.$ 
\end{comment}

Applying the transformation \eqref{eq37} with $m=2,$ \eqref{eq38} reduced to the following FODE
\begin{equation}\label{eq39}
\frac{d^\al\varphi}{d z^\al}=\bar{a}\varphi+\frac{\bar{b}}{\al}z\frac{d\varphi}{d z}+\frac{\bar{c}}{\al^2}z^2\frac{d^2\varphi}{d z^2}, \quad z>0,
\end{equation}
where $\bar{a}=a(a-1)a_2+a a_1+a_0,$  $\bar{b}=(p-2)\left(\frac{p-2}{\al}+2a-1+\frac{a_1}{a_2}\right)a_2$ and $\bar{c}=(p-2)^2a_2.$ 
From Proposition~\ref{le4} and Proposition~\ref{le6}, we have the following solutions of \eqref{eq39} for  $p\neq 2:$ 
\begin{enumerate}
\item For $0<\al<2,$
\begin{eqnarray*}
\varphi(z) & = & c_1 H_{1,2}^{2,0}\left[\frac{z^{-\al}}{(p-2)^2a_2}\biggr\vert\begin{array}{c}
\left(1,\al\right)\\
\left(-s_1,1\right), \left(-s_2,1\right)
\end{array}\right].
\end{eqnarray*}
\item For $\al>2,$ 
\begin{equation*}
\varphi(z)=\sum_{k=1}^n c_i z^{\al-k}{}_3\Psi_1\left[(p-2)^2 a_2 z^\al\biggr\vert\begin{array}{c}
\left(1-\frac{k}{\al}-s_1,1\right), \left(1-\frac{k}{\al}-s_2,1\right), (1,1)\\
(1+\al-k,\al)
\end{array}\right],
\end{equation*}
\end{enumerate}
where 
\begin{eqnarray*}
 s_1&=&\frac{1}{2(p-2)}\left(1-2a-\frac{a_1}{a_2}+\sqrt{\left(1-\frac{a_1}{a_2}\right)^2-\frac{4a_0}{a_2}}\right),\\
 s_2&=&\frac{1}{2(p-2)}\left(1-2a-\frac{a_1}{a_2}-\sqrt{\left(1-\frac{a_1}{a_2}\right)^2-\frac{4a_0}{a_2}}\right).
\end{eqnarray*}
If $a_0=\frac{a_2}{4}\left(\frac{a_1}{a_2}-\frac{p}{2}\right)\left(\frac{a_1}{a_2}+\frac{p}{2}-2\right)$ in \eqref{eq38}, then we can express these solutions in terms of the Wright function and the generalized Wright function ${}_2\Psi_1,$ respectively. In this case, by Corollary \ref{le7} we obtain the following solutions of \eqref{eq39}:
\begin{enumerate}
\item For $0<\al<2$,
\begin{equation*}
\varphi(z)=c_1z^{\frac{\alpha }{2}s}\Psi\left(-\frac{2z^{-\frac{\al}{2}}}{\left\vert p-2\right\vert\sqrt{a_2}};-\frac{\al}{2},1+\frac{\alpha }{2}s\right).
\end{equation*}
\item For $\alpha>2$,
\begin{equation*}
\varphi(z)=\sum_{k=1}^{n}c_kz^{\alpha-k}{}_2\Psi_{1}\left[\frac{(p-2)^2a_2z^\alpha}{4}\left\vert\begin{array}{c}
\left(2-\frac{2k}{\alpha}-s,2\right),(1,1)\\
(1+\alpha-k,\alpha)
\end{array}\right]\right.,
\end{equation*}
\end{enumerate}
where 
\begin{equation*}
s=\frac{1}{(p-2)}\left(\frac{p}{2}-2a-\frac{a_1}{a_2}\right).
\end{equation*}
If we set $a_0=a_1=p=0$ in the above solutions, then they correspond to the solutions obtained in \cite{luch} and \cite{luch1}.

If $p=2,$ then  \eqref{eq39} becomes 
\begin{equation*}
\frac{d^\al\varphi}{dz^\al}=(a(a-1)a_2+aa_1+a_0)\varphi, \quad z>0,
\end{equation*}
and the solution given in Proposition~\ref{le2} is
\begin{equation*}
\varphi(z)=\sum_{k=1}^nc_iz^{\al-k} E_{\al,1+\al-k}\left((a(a-1)a_2+a a_1+a_0)z^\al\right).
\end{equation*}
Finally, we can obtain the invariant solutions of \eqref{eq38} by substituting these solutions into \eqref{eq37}.
\subsection{Solutions to a system of fractional linear equations}
Let us consider the following system:
\begin{equation}\label{eq50}
\begin{cases}
\frac{\pal^\al u}{\pal t^\al} = a_1(x+c)^{m_1}v_x+b_1(x+c)^{m_1-1}v,\\
\frac{\pal^\al v}{\pal t^\al} = a_2(x+c)^{m_2}u_x+b_2(x+c)^{m_2-1}u,
\end{cases}
\quad \alpha>0,\quad x>-c,\quad t>0,
\end{equation}
where $a_1,a_2,b_1,b_2,m_1,m_2,c\in \mathbb{R}$ and $a_1a_2>0$. Then, with the substitution
\begin{equation*}
\begin{cases}
u(x,t) = (x+c)^{d+\frac{m_1}{2}}\varphi(z),\\
v(x,t) = (x+c)^{d+\frac{m_2}{2}}\psi(z),
\end{cases}
\mbox{ with } z=t(x+c)^{\frac{m_1+m_2-2}{2\al}} \mbox{ and } d\in\mathbb{R},
\end{equation*}
we obtain the system of FODEs
\begin{equation}\label{eq40}
\begin{cases}
\frac{d^\al\varphi}{d z^\al} = \bar{a}_1\psi+\frac{\bar{b}_1}{\al}z\frac{d\psi}{d z},\\
\frac{d^\al\psi}{d z^\al} = \bar{a}_2\varphi+\frac{\bar{b}_2}{\al}z\frac{d\varphi}{d z}, 
\end{cases}
\quad z>0,
\end{equation}
where
\begin{eqnarray*}
\bar{a}_1&=& \left(d+\frac{m_2}{2}\right)a_1+b_1,\qquad \bar{a}_2= \left(d+\frac{m_1}{2}\right)a_2+b_2,\\
\bar{b}_1&=&\frac{(m_1+m_2-2)a_1}{2},\qquad \bar{b}_2= \frac{(m_1+m_2-2)a_2}{2}.
\end{eqnarray*}

We obtain the following results by virtue of Proposition~\ref{le4} and Proposition~\ref{le6}.
\subsubsection*{Case 1.} Let us consider the case $m=m_1+m_2\neq 2.$ Then, for $0<\al<1$, the solution to \eqref{eq40} given in Proposition~\ref{le6} is 
\begin{eqnarray*}
\varphi(z)& = & c_1\sgn\left((m-2)a_1\right)H_{1,2}^{2,0}\left[\frac{z^{-2\al}}{(m-2)^2a_1a_2}\biggr\vert\begin{array}{c}
(1,2\al)\\
\left(\frac{1}{2}-\frac{\tilde{s}_1}{2},1\right), \left(-\frac{\tilde{s}_2}{2},1\right)
\end{array}\right],\\
\psi(z) & = & c_1 \sqrt{\frac{a_2}{a_1}}H_{1,2}^{2,0}\left[\frac{z^{-2\al}}{(m-2)^2a_1a_2}\biggr\vert\begin{array}{c}
(1,2\al)\\
\left(-\frac{\tilde{s}_1}{2},1\right), \left(\frac{1}{2}-\frac{\tilde{s}_2}{2},1\right)
\end{array}\right],
\end{eqnarray*}
and for $\al>1$, the solution given in the third assertion of Proposition~\ref{le4} is
\begin{align*}
\varphi(z)  =& \sum_{k=1}^n c_{k,1}z^{\al-k}{}_3\Psi_1\left[M^2 z^{2\al}\biggr\vert\begin{array}{c}
\left(1-\frac{k}{2\al}-\frac{\tilde{s}_1}{2},1\right), \left(\frac{1}{2}-\frac{k}{2\al}-\frac{\tilde{s}_2}{2},1\right), (1,1)\\
(1+\al-k,2\al)
\end{array}\right]\\
 & +Ma_1\sum_{k=1}^n c_{k,2}z^{2\al-k}{}_3\Psi_1\left[M^2 z^{2\al}\biggr\vert\begin{array}{c}
\left(\frac{3}{2}-\frac{k}{2\al}-\frac{\tilde{s}_1}{2},1\right), \left(1-\frac{k}{2\al}+\frac{\tilde{s}_2}{2},1\right), (1,1)\\
(1+2\al-k,2\al)
\end{array}\right],\\
\psi(z)  = &Ma_2 \sum_{k=1}^n c_{k,1}z^{2\al-k}{}_3\Psi_1\left[M^2 z^{2\al}\biggr\vert\begin{array}{c}
\left(1-\frac{k}{2\al}+\frac{\tilde{s}_1}{2},1\right), \left(\frac{3}{2}-\frac{k}{2\al}-\frac{\tilde{s}_2}{2},1\right), (1,1)\\
(1+2\al-k,2\al)
\end{array}\right]\\
 & +\sum_{k=1}^n c_{k,2}z^{\al-k}{}_3\Psi_1\left[M^2 z^{2\al}\biggr\vert\begin{array}{c}
\left(\frac{1}{2}-\frac{k}{2\al}-\frac{\tilde{s}_1}{2},1\right), \left(1-\frac{k}{2\al}-\frac{\tilde{s}_2}{2},1\right), (1,1)\\
(1+\al-k,2\al)
\end{array}\right],
\end{align*}
where
\begin{equation*}
M=m-2\qquad \tilde{s}_1 =-\frac{(2d+m_2)a_1+2b_1}{(m-2)a_1}, \qquad \tilde{s}_2 =-\frac{(2d+m_1)a_2+2b_2}{(m-2)a_2}.
\end{equation*}
\subsubsection*{Case 2.} Next, let us consider the case $m_1+m_2=2.$
In this case, we can rewrite \eqref{eq40} as
\begin{equation*}
\begin{cases}
\frac{d^\al\varphi}{dz^\al}=\bar{a}_1\psi,\\
\frac{d^\al\psi}{dz^\al}=\bar{a}_2\varphi,
\end{cases}
\quad z>0,
\end{equation*}
where 
\begin{equation*}
\bar{a}_1=(d+\frac{m_2}{2})a_1+b_1,\qquad \bar{a}_2=(d+\frac{m_1}{2})a_2+b_2,
\end{equation*}
and the solution given in Proposition~\ref{le2} is
\begin{eqnarray*}
\varphi(z) & = & \sum_{k=1}^nc_{k,1}z^{\al-k}E_{2\al,1+\al-k}(\bar{a}_1\bar{a}_2 z^{2\al})+\bar{a}_1\sum_{k=1}^nc_{k,2}z^{2\al-k}E_{2\al,1+2\al-k}(\bar{a}_1\bar{a}_2z^{2\al}),\\
\psi(z) & = & \bar{a}_2\sum_{k=1}^nc_{k,1}z^{\al-k}E_{2\al,1+2\al-k}(\bar{a}_1\bar{a}_2 z^{2\al})+\sum_{k=1}^nc_{k,2}z^{2\al-k}E_{2\al,1+\al-k}(\bar{a}_1\bar{a}_2z^{2\al}).
\end{eqnarray*}
Then, similarly to the previous cases, using the solutions of the reduced system, we are able to obtain invariant solutions of (\ref{eq50}).

Readers interested in more applications of the results obtained in this work are referred to works \cite{our1} and \cite{our2}.
\section{Conclusion}
We have systematically presented several solutions expressed in terms of Mittag-Leffler functions, generalized Wright functions and Fox H-functions, to the fractional linear differential equation \eqref{eq1} and system of fractional linear differential equations \eqref{eq2}. Because first-order and second-order differential equations are the most commonly used differential equations in the modeling of systems in natural and social sciences, we studied the solutions of fractional linear equations with integer-order derivatives up to second order in detail. In addition, the solutions of fractional linear equations  and systems of equations with integer-order derivatives of higher-orders were also treated. The results obtained in this case can be easily proved by following the steps in the proofs given here for the case of second-order derivatives.

In this work, we have merely presented several solutions of the studied equations and systems. Because we have expressed these solutions in terms of well-known special functions, we have not studied their properties. Also, we have not sought a general method for solving equations of the type considered in this paper. This is a topic of future study. 
\section*{Acknowledgements}
This work was supported by JSPS (KAKENHI Grant No. 15H03613) and by the Science and Technology Foundation of Mongolia (Grant No. SSA-012/2016).

%%%%%%%%%% References %%%%%%%%%%%%%%%%%%%%%%%%%%%%%%%%
%%%% arranged in ALPHABETIC ORDER of Authors' Families
%%%% for articles, insert also DOI numbers if available

  %%%%%%%%%%%%%%%%%%%%%%%%%%%%%%%%

%%%%%%%%%% put authors' addresses here, in \it %%%%%%%%

 \bigskip \smallskip

 \it

 \noindent
%(First) Author's full postal address
$^1$ Graduate School of Mathematics \\
Kyushu University, \\
744 Motooka, Fukuoka 819-0395, Japan\\[4pt]
e-mail: k-dorujigotofu@math.kyushu-u.ac.jp
%\hfill Received: November 1, 2014 \\[12pt]
% Second Author's address

$^2$ Institute of Mathematics for Industry\\
Kyushu University,\\
744 Motooka, Fukuoka 819-0395, Japan \\[4pt]
e-mail: ochiai@imi.kyushu-u.ac.jp

$^3$ Department of Mathematics,\\
National University of Mongolia, \\
P.B. 507/38, Chingiltei 6, Ulaanbaatar\\ 15141, Mongolia\\
e-mail: zunderiya@gmail.com
\end{document}